\documentclass[a4paper]{amsart}
\usepackage[active]{srcltx}

\newcommand{\Sets}{{\mathcal{S}ets}}
\newcommand{\sSets}{{s\mathcal{S}ets}}
\usepackage{amssymb}
\usepackage{amsmath}
\usepackage[all]{xy}
\SelectTips{cm}{} 

\newcommand{\inn}{\mathop{\rm in}} 
\newcommand{\Hom}{\mathop{\rm Hom}}
\newcommand{\Oper}{\mathop{\rm Oper}} 
\newcommand{\Coll}{\mathop{\rm Coll}}
\newcommand{\End}{\mathop{\rm End}}

\newcommand{\V}{\mathcal{V}}
\newcommand{\E}{\mathcal{E}}
\newcommand{\Sp}{{\mathcal{S}p}^{\Sigma}}

\SelectTips{cm}{}
\theoremstyle{plain}
\newtheorem{theorem}{Theorem}[section]
\newtheorem{corollary}[theorem]{Corollary}
\newtheorem{lemma}[theorem]{Lemma}
\newtheorem{proposition}[theorem]{Proposition}

\theoremstyle{definition}
\newtheorem{definition}[theorem]{Definition}

\theoremstyle{remark}
\newtheorem{remark}[theorem]{Remark}
\newtheorem{example}[theorem]{Example}
\numberwithin{equation}{section} \numberwithin{figure}{section}

\begin{document}
\title[A model structure for coloured operads in symmetric spectra]{A model structure for coloured operads \\ in symmetric spectra}
\author[J. J. Guti\'errez]{Javier J. Guti\'errez}
\address{Centre de Recerca Matem\`atica, Apartat 50, 08193 Bellaterra, Spain}
\email{jgutierrez@crm.cat}
\author[R. M. Vogt]{Rainer M. Vogt}
\address{Universit\"at Osnabr\"uck, Fachbereich Mathematik/Informatik, Albrechtstr.\ 28, 
49069 Osnabr\"uck, Germany}
\email{rainer@mathematik.uni-osnabrueck.de}


\begin{abstract}
We describe a model structure for coloured operads with values in the category of symmetric
spectra (with the positive model structure), in which fibrations and
weak equivalences are defined at the level of the underlying collections. 
This allows us to treat $R$-module spectra (where $R$ is a cofibrant ring spectrum) as algebras over a cofibrant spectrum-valued operad with $R$ as its first term. Using this model structure, we give sufficient conditions for homotopical localizations in the category of symmetric spectra to preserve module structures.
\end{abstract}

\maketitle

\section{Introduction}
In a series of papers \cite{BM03}, \cite{BM06}, \cite{BM07}, Berger and Moerdijk studied the homotopy theory of operads and coloured operads in monoidal model categories from an axiomatic point of view. For any cofibrantly generated monoidal model category $\V$ satisfying certain conditions, they described a model structure for the category of $C$\nobreakdash-coloured operads (for a fixed set of colours $C$) in which the fibrations and the weak equivalences were defined at the level of the underlying collections. The model structure was, in fact, transferred from the model structure on the category of $C$\nobreakdash-coloured collections $ \Coll_C(\V)$ via the free-forgetful adjunction
$$
\xymatrix{
F : \Coll_C(\V) \ar@<3pt>[r] & \ar@<3pt>[l] {\rm Oper}_C(\V) : U.
}
$$

This model structure is very useful for proving general results concerning constructions with
coloured operads. In this way, one can define, for an operad $P$, the notion of homotopy $P$-algebra
as an algebra over a cofibrant resolution of $P$. One can also define a
generalization of the $W$-construction of Boardman-Vogt in monoidal
model categories equipped with an interval \cite{BM06}. Berger and Moerdijk also proved a
generalization of the homotopy invariance property of algebras over
cofibrant operads, extending results of \cite{BV73}.

The conditions imposed by Berger and Moerdijk on $\V$  concern the cofibrancy of the unit of the monoidal structure, the existence of a symmetric monoidal fibrant replacement functor, and the existence of a coalgebra interval with a cocommutative comultiplication \cite[Theorem 2.1]{BM07}.  These conditions hold for the categories of simplicial sets, compactly generated spaces and chain complexes over any commutative ring. However, they do not hold for (symmetric) spectra. In the category of symmetric spectra with the positive model structure, the unit spectrum $S$ is not cofibrant and the existence of a monoidal fibrant replacement functor is not known. In fact, by an argument of Lewis~\cite{Lew91}, no monoidal model category of spectra can simultaneously have a cofibrant unit and a symmetric monoidal fibrant replacement functor.

One can try to avoid the cofibrancy condition on the unit by enriching the base category over another monoidal model category in which the unit and the intervals are nicer~\cite[Theorem 2.4]{Kro}. In his work, Kro considered the category of orthogonal spectra with the positive model structure enriched over compactly generated topological spaces, and he proved that the categories of (one-coloured) reduced operads and positive operads in the category of orthogonal spectra admit a transferred model structure. We cannot expect to apply his argument to the category of symmetric spectra, since a symmetric monoidal fibrant replacement is still needed and Kro's functor for orthogonal spectra is not valid for symmetric spectra \cite[Remark 3.4]{Kro}. 

Here we propose an alternative approach suitable for the category of $C$\nobreakdash-co\-lou\-red operads with values in symmetric spectra, where $C$ is any set of colours. Given a symmetric monoidal category $\V$, operads in $\V$ act on any monoidal $\V$\nobreakdash-category $\E$.  For a fixed set of colours $C$ and any monoidal $\V$\nobreakdash-category $\E$, we describe a coloured operad in $\V$ acting on $\E$ whose algebras are $C$\nobreakdash-co\-lou\-red operads in $\E$. The category of symmetric spectra with the positive model structure is a monoidal model category enriched and tensored over the model category of simplicial sets.  Combining this with the fact that for any coloured operad $P$ in simplicial sets acting on symmetric spectra there is a model structure on the category of $P$-algebras~\cite[Theorem 1.3]{EM06}, we produce a model structure for the category of $C$\nobreakdash-coloured operads in symmetric spectra in which fibrations and weak equivalences are defined at the level of the underlying collections.

We use this model structure to prove that enriched homotopical localizations in the category of symmetric spectra preserve algebras over cofibrant operads. In particular, they preserve $R$-module spectra, where $R$ is a cofibrant ring spectrum. 
Our proof simplifies considerably the one given in \cite[Theorem 6.1]{CGMV} (which used a coloured operad with two colours in the category of simplicial sets acting on symmetric spectra), since we are now allowed to use spectrum-valued operads.  
More concretely, for any ring spectrum $R$ there is an operad $P_R$ with $P_R(1)=R$ and $0$ otherwise, whose algebras are the $R$-modules. This operad is cofibrant if $R$ itself is a cofibrant ring spectrum. Thus, if $L$ is an enriched homotopical localization functor in the category of symmetric spectra and $M$ is an $R$-module, where $M$ is cofibrant (as a spectrum) and $R$ is a cofibrant ring spectrum, then $LM$ has a homotopy unique $R$-module structure such that the localization map $M\longrightarrow LM$ is a map of $R$-modules.

\section{Coloured operads and algebras in enriched categories}
Coloured operads may be defined in any symmetric monoidal category $\V$ and algebras over them
make sense in any symmetric monoidal category $\E$ enriched over $\V$. In this first section, we review
some terminology of enriched categories \cite[\S 6]{Bor94},  \cite[\S 4.1]{Hov99}, \cite{Kel82}, and recall the definition of coloured operads and their algebras
\cite{BV73}, \cite[\S 2]{EM06}, \cite{Lei}.

Throughout the paper, $\V$ will denote a cocomplete closed
symmetric mo\-noi\-dal category with unit $I$, tensor product $\otimes$ and internal. We will denote by $0$ an initial object of $\V$.

A functor $F\colon \V\longrightarrow \V'$ between symmetric monoidal categories is called \emph{symmetric monoidal} if it is equipped
with a unit $I_{\V'}\longrightarrow F(I_{\V})$ and a binatural transformation $F(-)\otimes_{\V'} F(-)\longrightarrow F(-\otimes_{\V} -)$ satisfying the usual associativity, symmetry and unit conditions. A symmetric monoidal functor is called  \emph{strong} if the structure maps are isomorphisms.

\subsection{Enriched categories}
A \emph{category enriched over $\V$} is a category $\E$ together
with an enrichment functor
$$
\Hom\nolimits_{\E}(-,-)\colon \E^{\rm op}\times \E\longrightarrow \V
$$
and, for every $X$, $Y$ and $Z$ in $\E$, composition morphisms 
$$
\Hom\nolimits_{\E}(Y,Z)\otimes \Hom\nolimits_{\E}(X,Y)\longrightarrow \Hom\nolimits_{\E}(X,Z)
$$
satisfying the associativity law, and morphisms
$$
I\longrightarrow \Hom\nolimits_{\E}(X,X)
$$
which are left and right identities for the composition morphisms.

A \emph{$\V$-module category} $\E$ is a category
enriched and tensored over $\V$, i.e., equipped with a functor
$$
-\otimes -\colon \V\times \E\longrightarrow \E
$$
such that the following two conditions hold:
\begin{itemize}
\item[(i)] There are natural isomorphisms
\begin{equation}
A\otimes(B\otimes X)\cong (A\otimes B)\otimes X, \quad I\otimes X\cong X
\label{nat_iso1}
\end{equation}
for every $A$ and $B$ in $\V$ and every $X$ in $\E$, rendering certain coherence diagrams commutative (see \cite[\S 4.1]{Hov99} for details).
\item[(ii)] There are natural isomorphisms
\begin{equation}
\E(A\otimes X, Y)\cong \V(A, \Hom\nolimits_{\E}(X,Y))
\label{natiso}
\end{equation}
for every $X$ and $Y$ in $\E$, and every $A$ in $\V$.
\end{itemize}
In particular, when $A=I$, the condition $\E(X,Y)\cong\V(I,\Hom_{\E}(X,Y))$ holds for all $X$, $Y$. Note that~(\ref{natiso}) implies that the enriched version of this adjunction also holds, that is, there are natural isomorphisms
$$
\Hom\nolimits_{\E}(A\otimes X, Y)\cong \Hom\nolimits_{\V}(A, \Hom\nolimits_{\E}(X,Y))
$$
for every $X$ and $Y$ in $\E$, and every $A$ in $\V$, where $\Hom_{\V}$ denotes the internal hom of $\V$. To prove this, it is enough to check that there are natural isomorphisms
$$
\V(W, \Hom\nolimits_{\E}(A\otimes X, Y))\cong\V(W,\Hom\nolimits_{\V}(A, \Hom\nolimits_{\E}(X,Y))),
$$
for every $W$ in $\V$, and this follows from~(\ref{natiso}) and~(\ref{nat_iso1}).
 
A \emph{monoidal $\V$-module category} $\E$ is a symmetric monoidal category that is also a $\V$-module category and such that
the $\V$-action commutes with the monoidal product of $\E$. That is, there are natural isomorphisms
\begin{equation}
A\otimes (X\otimes Y)\cong (A\otimes X)\otimes Y
\label{nat_iso2}
\end{equation}
for every $X$ and $Y$ in $\E$ and every $A$ in $\V$; see~\cite[\S1.1.2, \S1.1.12]{Fre09}. We will refer to monoidal $\V$-module categories as 
\emph{monoidal $\V$-categories}. Any closed symmetric monoidal category $\V$ is enriched over itself and a monoidal $\V$-category.

\begin{lemma}
If $\E$ is a monoidal $\V$-category, then the functor from $\V$ to $\E$ sending every object $A$ in $\V$ to  $\widetilde{A}=A\otimes I_{\E}$, where $I_{\E}$ denotes the unit of the monoidal structure of $\E$, is strong symmetric monoidal. 
\label{strong_monoidal}
\end{lemma}

\begin{proof}
Using the natural isomorphisms~(\ref{nat_iso1}) and~(\ref{nat_iso2}), we have that  $\widetilde{I}\cong I_{\E}$ and that
\begin{multline}\notag
\widetilde{A}\otimes\widetilde{B}\cong(A\otimes I_{\E})\otimes (B\otimes I_{\E})\cong
A\otimes(I_{\E}\otimes (B\otimes I_{\E}))\cong \\
A\otimes(B\otimes I_{\E})\cong
(A\otimes B)\otimes I_{\E}\cong \widetilde{A\otimes B}
\end{multline}
for every $A$ and $B$ in $\V$.
\end{proof}

\subsection{Coloured operads}
Let $C$ be any set, whose elements will be called \emph{colours}. A
\emph{$C$\nobreakdash-coloured collection} $K$ in $\V$ consists of a set of
objects $K(c_1,\ldots, c_n;c)$ in $\V$ for each $(n+1)$-tuple of
colours $(c_1,\ldots, c_n;c)$ equipped with a right action of the
symmetric group $\Sigma_n$ by means of maps
$$
\alpha^*\colon K(c_1,\ldots, c_n;c)\longrightarrow
K(c_{\alpha(1)},\ldots, c_{\alpha(n)};c),
$$
where $\alpha\in\Sigma_n$ (if $n=0$ or $n=1$, then $\Sigma_n$ is the
trivial group).

A \emph{morphism} of $C$\nobreakdash-coloured collections $\varphi\colon
K\longrightarrow K'$ is a family of maps
$$
\varphi_{c_1,\ldots,c_n;c}\colon K(c_1,\ldots,c_n;c)\longrightarrow
K(c_1,\ldots, c_n; c)
$$
in $\V$, ranging over all $n\ge 0$ and all $(n+1)$-tuples
$(c_1,\ldots, c_n;c)$, and compatible with the action of the
symmetric groups. The category of $C$\nobreakdash-coloured collections in $\V$
is denoted by $\Coll_C(\V)$.

\begin{definition}
A \emph{$C$\nobreakdash-coloured operad} $P$ in $\V$ is a $C$\nobreakdash-coloured
collection equip\-ped with unit maps $I\longrightarrow P(c;c)$ for
every $c\in C$ and, for every $(n+1)$-tuple of colours
$(c_1,\ldots, c_n;c)$ and $n$ given tuples
\[
(a_{1,1},\ldots, a_{1,k_1}; c_1),\ldots, (a_{n,1},\ldots, a_{n,k_n};
c_n),
\]
a \emph{composition product} map
$$\xymatrix{
P(c_1,\ldots, c_n;c)\otimes P(a_{1,1},\ldots,
a_{1,k_1};c_1)\otimes\cdots\otimes P(a_{n,1},\ldots,
a_{n,k_n};c_n)\ar[d]
\\ P(a_{1,1},\ldots,a_{1,k_1},a_{2,1},\ldots,a_{2,k_2},\ldots,a_{n,1},\ldots,a_{n,k_n};c),
} $$ compatible with the action of the symmetric groups and subject
to associativity and unitary compatibility relations; see, for
example, \cite[\S 2]{EM06}.
\end{definition}

\begin{remark} 
Alternatively, one can replace the composition product
in the definition of a $C$\nobreakdash-coloured operad by the \emph{$\circ_i$-operations}
$$
\xymatrix{
P(c_1,\ldots, c_i,\ldots,c_n;c)\otimes P(a_1,\ldots,a_m;c_i) \ar[d]^{\circ_i}\\
 P(c_1,\ldots,c_{i-1},a_1,\ldots,a_m,c_{i+1},\ldots,c_n;c) }
$$
for every $(n+1)$-tuple $(c_1,\ldots, c_n, c)$, every
$m$-tuple $(a_1,\ldots, a_m)$, and all $1\leq i\leq n$. The
$\circ_i$-operations are compatible with the action of
$\Sigma_n$ and are subject to the usual associativity and unitary
compatibility relations; see, for example, \cite[\S 2.1]{Lei}. We
will make use of both definitions in the next section.
\end{remark}

A morphism of $C$\nobreakdash-coloured operads is a morphism of the
underlying $C$\nobreakdash-co\-lou\-red collections that is compatible
with the unit maps and the composition product maps (or the $\circ_i$-operations).
The category of $C$\nobreakdash-coloured operads in $\V$ will be denoted by ${\rm
Oper}_C(\V)$. There is a free-forgetful adjunction
\begin{equation}
\xymatrix{
F : \Coll_C(\V) \ar@<3pt>[r] & \ar@<3pt>[l] {\rm Oper}_C(\V) : U
}
\label{freeforgetful}
\end{equation}
where $U$ is the forgetful functor, and the left adjoint is the free coloured operad generated by
a collection.

\subsection{Algebras over coloured operads}
\label{algebras}
Algebras over operads can be
defined in any monoidal $\V$-category $\E$, as
follows. Let $\E^{C}$ denote the product category $\prod_{c\in C}\E$
indexed by the set of colours $C$. For every object ${\bf
X}=(X(c))_{c\in C}\in \E^C$, the \emph{endomorphism $C$\nobreakdash-coloured
operad} $\End({\bf X})$ in $\V$ associated with ${\bf X}$ is defined
by
$$
\End({\bf X})(c_1,\ldots, c_n;
c)=\Hom\nolimits_{\E}(X(c_1)\otimes\cdots\otimes X(c_n);X(c))
$$
where $X(c_1)\otimes\cdots\otimes X(c_n)$ is meant to be the unit of the monoidal category $\E$ if
$n=0$. The composition product is ordinary composition and the
$\Sigma_n$-action is defined by permutation of the factors.

\begin{definition}
Let $P$ be any $C$\nobreakdash-coloured operad in $\V$. An \emph{algebra over $P$} or a \emph{P-algebra} in
$\E$ is an object ${\bf X}=(X(c))_{c\in C}$ of $\E^C$ together with
a morphism
$$
P\longrightarrow \End({\bf X})
$$
of $C$\nobreakdash-coloured operads in $\V$.
\end{definition}
Equivalently, a $P$-algebra in $\E$ can be defined as a family of objects $X(c)$ in
$\E$ for all $c\in C$, together with maps
$$
P(c_1,\ldots, c_n;c)\otimes X(c_1)\otimes\cdots\otimes
X(c_n)\longrightarrow X(c)
$$
for every $(n+1)$-tuple $(c_1,\dots,c_n;c)$, compatible
with the symmetric group action, the units of $P$, and subject to the usual associativity relations.

A \emph{map of $P$-algebras} $\mathbf{f}\colon\mathbf{X}\longrightarrow \mathbf{Y}$ is a family of maps
$(f_c\colon X(c)\longrightarrow Y(c))_{c\in C}$ in $\mathcal{E}$ such that the diagram of $C$-coloured collections
$$
\xymatrix{
P\ar[r] \ar[d] & \End(\mathbf{X}) \ar[d] \\
\End(\mathbf{Y}) \ar[r]& \Hom(\mathbf{X},\mathbf{Y})
}
$$
commutes, where the top and left arrows are the $P$-algebra structures on $\mathbf{X}$ and $\mathbf{Y}$ respectively, and the $C$-coloured collection $\Hom(\mathbf{X},\mathbf{Y})$ is defined as
$$
\Hom(\mathbf{X},\mathbf{Y})(c_1,\ldots, c_n;c)=\Hom\nolimits_{\mathcal{E}}(X(c_1)\otimes\cdots\otimes X(c_n), Y(c)).
$$
The arrows ${\rm End}({\bf X})\longrightarrow {\rm Hom}({\bf X},
{\bf Y})$ and ${\rm End}({\bf Y})\longrightarrow {\rm Hom}({\bf X},
{\bf Y})$ are induced by ${\bf f}$ by composing on each side.

If the category $\mathcal{V}$ has pullbacks, then a map $\mathbf{f}$ of $P$-algebras can be seen as a map of $C$-coloured operads
$$
P\longrightarrow \End(\mathbf{f}),
$$
where the $C$-coloured operad $\End(\mathbf{f})$ is obtained as the pullback of the diagram of $C$-coloured collections
$$
\xymatrix{
\End(\mathbf{f})\ar@{.>}[r] \ar@{.>}[d] & \End(\mathbf{X}) \ar[d] \\
\End(\mathbf{Y}) \ar[r]& \Hom(\mathbf{X},\mathbf{Y}).
}
$$
The $C$\nobreakdash-coloured collection ${\rm End}({\bf f})$ inherits
indeed a $C$\nobreakdash-coloured operad structure in $\V$ from the
$C$\nobreakdash-coloured operads ${\rm End}({\bf X})$ and ${\rm
End}({\bf Y})$, as observed in \cite[Theorem~3.5]{BM03}. We will denote the category of $P$-algebras in $\E$ by ${\rm Alg}_P(\E)$.

For any monoidal $\mathcal{V}$-category $\E$, the functor from $\V$ to $\E$ sending each object $A$ in $\V$ to  $\widetilde{A}=A\otimes I_{\E}$ is strong symmetric monoidal (Lemma~\ref{strong_monoidal}). Hence, it sends $C$\nobreakdash-coloured operads in $\V$ to $C$\nobreakdash-coloured operads in $\E$. Moreover, if $P$ is a $C$\nobreakdash-coloured operad in $\V$, then an object ${\bf X}$ of $\E^C$ is a $P$-algebra if and only if it is a $\widetilde{P}$-algebra.

\section{Coloured operads as algebras}
\label{colopasalg}
The following is an example of a coloured operad whose algebras are $C$\nobreakdash-co\-lou\-red operads
for a fixed set of colours $C$. It is initially constructed in the category of sets, but can be transported
to an arbitrary symmetric monoidal category via the strong symmetric monoidal functor that sends a set to a coproduct of copies
of the unit of the monoidal category (indexed by the elements of the set). The description of this operad
is made in terms of trees, so we need to introduce some terminology on them first.

\subsection{Trees}
A \emph{tree} $T$ is a connected finite graph with no loops. 
We denote the set of vertices of $T$ by $V(T)$ and the set of edges of $T$ by $E(T)$. Edges are directed. There are two different types of edges: the ones with a vertex at both ends, called \emph{inner edges}, and the ones with
a vertex only at one end or with no vertices, called \emph{external edges}. 
A \emph{rooted tree} is a tree where each vertex $v$ has exactly one outgoing edge denoted by ${\rm out}(v)$ and a set ${\rm in}(v)$ of incoming edges whose elements are called \emph{inputs} of $v$ (note that ${\rm in}(v)$ can be empty). The cardinality of
${\rm in}(v)$ is called the \emph{valence} of $v$. Consequently, in our trees there is a unique external edge leaving a vertex; we call it the \emph{root} or
\emph{output edge} of $T$. The other external edges form the set ${\rm in}(T)$ of \emph{input edges} or \emph{leaves} of $T$.

A \emph{planar rooted tree} is a rooted tree $T$ together with a linear ordering of
$\inn(v)$ for each vertex $v$ of $T$.

When drawing rooted trees in the plane, we represent them as oriented towards the
output, drawn at the bottom, with the canonical orientation from the leaves towards the root. In this case,
the linear order appears from reading the incoming edges from left to right.

\begin{definition}
Let $C$ be any set. A \emph{planar rooted $C$\nobreakdash-coloured tree} is a
planar rooted tree $T$ together with a function $c_T\colon E(T)\longrightarrow C$, called
a \emph{colouring function}.
\end{definition}

\subsection{The coloured operad $S^C$ in sets}
Let $C$ be a set of colours. We define a $D$-coloured collection
$S^{C}$ in $\Sets$, where
$$
D=\{(c_1,\ldots, c_n; c)\mid c_i, c\in C, \, n\ge 0\}.
$$
This collection can be endowed with a $D$-coloured operad structure
whose algebras are $C$\nobreakdash-coloured operads in $\Sets$ as follows. We use the following notation for the elements of the set $D$:
$$
\overline{c}_i=(c_{i,1},\ldots, c_{i,k_i}; c_i) \,\mbox{ and }\,
\overline{a}=(a_1,\ldots, a_m;a).
$$
For each $(n+1)$-tuple of
colours $(\overline{c}_1,\ldots, \overline{c}_n; \overline{a})$, the
elements of $S^{C}(\overline{c}_1,\ldots, \overline{c}_n;
\overline{a})$ are equivalence classes of triples $(T,\sigma,\tau)$,
where:
\begin{itemize}
\item[{\rm (i)}] $T$ is a planar rooted $C$\nobreakdash-coloured tree with $m$ input edges coloured by $a_1,\ldots, a_m$, a root edge coloured by
$a$, and $n$ vertices.
\item[{\rm (ii)}] $\sigma$ is a bijection $\sigma\colon\{1,\ldots,n\}\longrightarrow V(T)$ with the property that
$\sigma(i)$ has $k_i$ input edges coloured from left to right by $c_{i,1},\ldots, c_{i,k_i}$ and one output edge
coloured by $c_i$.
\item[{\rm (iii)}] $\tau$ is a bijection $\tau\colon\{1,\ldots, m\}\longrightarrow \inn(T)$ such that $\tau(i)$ has
colour $a_i$.
\end{itemize}
Two such triples $(T,\sigma, \tau)$, $(T'\sigma',\tau')$ are equivalent if and only if there is a planar
isomorphism $\varphi\colon T\longrightarrow T'$ such that $\varphi\circ\sigma=\sigma'$, $\varphi\circ \tau=\tau'$ and
$\varphi$ respects the colouring, i.e., if $e$ is an edge of $T$ of colour $c$, then the edge $\varphi(e)$ in $T'$
has colour $c$ too.

\begin{example}
If $C=\{a,b,c\}$, then  an element of
$$
S^C((a,b;c),(b,b;a),(\,\,\,;a),(c,a,a;b); (b,b,a,c;c))
$$
is represented, for example, by a tree
$$
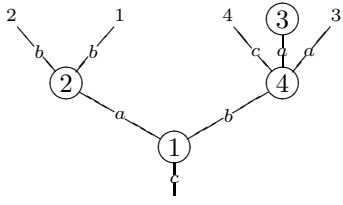
\xymatrixcolsep{1pc}
\xymatrixrowsep{1pc}
\entrymodifiers={=<1pc>[o][F-]} \xymatrix{
*{^2}\ar@{-}[dr]|b & *{} & *{^1}\ar@{-}[dl]|b & *{} & *{^4}\ar@{-}[dr]|c & 3 \ar@{-}[d]|a& *{^3} \ar@{-}[dl]|a\\
*{} & 2 \ar@{-}[drr]|a & *{} & *{} & *{} & 4\ar@{-}[dll]|b & *{} \\
*{} & *{} & *{} & 1\ar@{-}[d]|-c & *{} & *{} & *{} \\
*{} & *{} & *{} & *{} & *{} & *{} & *{}
}
$$
\end{example}

Any permutation $\alpha\in\Sigma_n$ induces a map
$$
\alpha^*\colon S^C(\overline{c}_1,\ldots, \overline{c}_n; \overline{a})\longrightarrow
S^C(\overline{c}_{\alpha(1)},\ldots, \overline{c}_{\alpha(n)}; \overline{a})
$$
that sends $(T,\sigma,\tau)$ to $(T, \sigma\circ\alpha, \tau)$. That is, $\alpha^*(T)$ is the same tree as
$T$ but with a renumbering of the vertices given by $\alpha$.

Let $\alpha$ be any element in $\Sigma_m$ and $\overline{a}_1=(a_{\alpha(1)},\ldots, a_{\alpha(m)}; a)$.
Then the set $S^C(\overline{a}_1; \overline{a})$ can be identified with the subset of elements of $\Sigma_m$
that permute the element $(a_{\alpha(1)},\ldots, a_{\alpha(m)})$ into $(a_1,\ldots, a_m)$. In particular, if
$\overline{a}=\overline{a}_1$ then the set $S^C(\overline{a}; \overline{a})$ can be identified with
the (opposite) subgroup of $\Sigma_m$ that leaves the colours $a_1,\ldots, a_m$ invariant.

\begin{proposition}
The $D$-coloured collection $S^C$ admits a $D$-coloured operad structure in $\Sets$.
\end{proposition}
\begin{proof}
There is a distinguished element $1_{\overline{a}}$ in $S^C(\overline{a};\overline{a})$
corresponding to the tree
$$
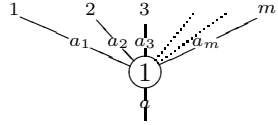
\xymatrixcolsep{1pc}
\xymatrixrowsep{1pc}
\entrymodifiers={=<1pc>[o][F-]}
\xymatrix{
*{^1}\ar@{-}[drrr]|{a_1}b &  *{} & *{^2}\ar@{-}[dr]|{a_2} & *{^3}\ar@{-}[d]|{a_3}
& *{}\ar@{.}[dl] & *{}\ar@{.}[dll] & *{^m}\ar@{-}[dlll]|{a_m}\\
*{} & *{} & *{} & 1\ar@{-}[d]|-a & *{} & *{} & *{} & *{}\\
*{} & *{} & *{} & *{} & *{} & *{} & *{} & *{}
}
$$
for every $\overline{a}\in D$. These elements will be the
units of the coloured operad $S^C$.

The composition product on $S^C$ is defined as follows. Given an
element $(T,\sigma,\tau)$ of
$S^C(\overline{c}_1,\ldots,\overline{c}_n; \overline{a})$ and $n$
elements $(T_1,\sigma_1,\tau_1),\ldots,(T_n,\sigma_n,\tau_n)$ of
$$
S^C(\overline{d}_{1,1},\ldots,\overline{d}_{1,k_1}; \overline{c}_1),\ldots,
S^C(\overline{d}_{n,1},\ldots,\overline{d}_{n,k_n}; \overline{c}_n)
$$
respectively, we obtain an element $T'$ of
$$
S^C(\overline{d}_{1,1},\ldots,\overline{d}_{1,k_1},\overline{d}_{2,1},\ldots,\overline{d}_{2,k_2},\ldots,
\overline{d}_{n,1},\ldots,\overline{d}_{n,k_n} ; \overline{a})
$$
in the following way:
\begin{itemize}
\item[({\rm i})] $T'$ is obtained by replacing the vertex $\sigma(i)$ of $T$ by the tree $T_i$ for every~$i$.
This is done by identifying the input edges of $\sigma(i)$ in $T$
with the input edges if $T_i$ via the bijection $\tau_i$. The $c_{i,j}$-coloured input edge of $\sigma(i)$ is
matched with the $c_{i,j}$-coloured input edge $\tau_i(j)$ of $T_i$. (Note
that the colours of the input edges and the output of $\sigma(i)$ coincide with the colours of
the input edges and the root of $T_i$.)
\item[({\rm ii})] The vertices of $T'$ are numbered following the order, i.e., first number the subtree
$T_1$ in $T'$ ordered by $\sigma_1$, then $T_2$ ordered by $\sigma_2$ and so on.
\item[({\rm iii})] The input edges of $T'$ are numbered following $\tau$ and the identifications given by $\tau_i$.
\end{itemize}
This composition product is associative and compatible with the units and the action of the symmetric group. 
\end{proof}

\begin{example}
Let $C=\{a,b,c\}$ as before and let $T$ be an element of
$$
S^C((a,b;c),(c, b;a),(a,a,a;b);(c,b,a,a,a;c))
$$
represented by the tree
$$
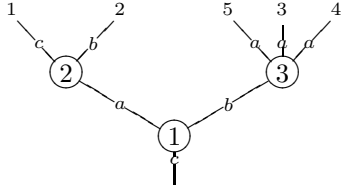
\xymatrixcolsep{1pc}
\xymatrixrowsep{1pc}
\entrymodifiers={=<1pc>[o][F-]} \xymatrix{
*{^1}\ar@{-}[dr]|c & *{} & *{^2}\ar@{-}[dl]|b & *{} & *{^5}\ar@{-}[dr]|a & *{^3}\ar@{-}[d]|a& *{^4} \ar@{-}[dl]|a\\
*{} & 2 \ar@{-}[drr]|a & *{} & *{} & *{} & 3\ar@{-}[dll]|b & *{} \\
*{} & *{} & *{} & 1\ar@{-}[d]|c & *{} & *{} & *{} \\
*{} & *{} & *{} & *{} & *{} & *{} & *{}
}
$$
and $T_1$, $T_2$ and $T_3$ be elements of
$$
S^C((a,b;c),(c;c);(a,b;c)),\, S^C((b,b;a),(c;b);(c,b;a)),
$$
$$
\mbox{ and }
S^C((a,a;c),(a,c;b);(a,a,a;b))
$$
respectively, represented by the trees
$$
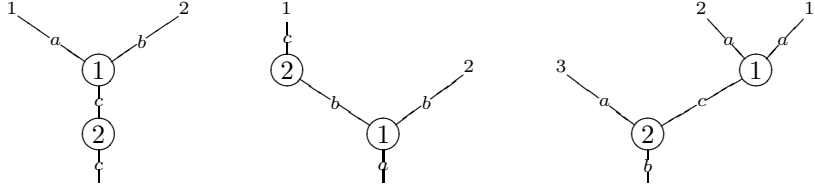
\begin{minipage}{4cm}
\xymatrixcolsep{1pc}
\xymatrixrowsep{1pc}
\entrymodifiers={=<1pc>[o][F-]} \xymatrix{
*{^1} \ar@{-}[drr]|a & *{} & *{} & *{} & *{^2}\ar@{-}[dll]|b \\
*{} & *{} & 1\ar@{-}[d]|c & *{} & *{} \\
*{} & *{} & 2\ar@{-}[d]|-c & *{} & *{} \\
*{} & *{} & *{} & *{} & *{}
}
\end{minipage}
\qquad\quad
\begin{minipage}{4cm}
\xymatrixcolsep{1pc}
\xymatrixrowsep{1pc}
\entrymodifiers={=<1pc>[o][F-]} \xymatrix{
*{^1}\ar@{-}[d]|-c & *{} & *{} & *{} & *{} \\
2 \ar@{-}[drr]|b & *{} & *{} & *{} & *{^2}\ar@{-}[dll]|b \\
*{} & *{} & 1\ar@{-}[d]|-a & *{} & *{} \\
*{} & *{} & *{} & *{} & *{}
}
\end{minipage}
\qquad\quad
\begin{minipage}{4cm}
\xymatrixcolsep{1pc}
\xymatrixrowsep{1pc}
\entrymodifiers={=<1pc>[o][F-]} \xymatrix{
*{} & *{} & *{} & *{^2}\ar@{-}[dr]|a &*{} & *{^1}\ar@{-}[dl]|a \\
*{^3} \ar@{-}[drr]|a & *{} & *{} & *{} & 1 \ar@{-}[dll]|c & *{} \\
*{} & *{} & 2\ar@{-}[d]|-b & *{} & *{} &*{} \\
*{} & *{} & *{} & *{} & *{} & *{}
}
\end{minipage}
$$
Applying the composition product, we get an element in
$$
S^C((a,b;c),(c;c),(b,b;a),(c;b),(a,a;c),(a,c;b);(c,b,a,a,a;c))
$$
that is represented by the following tree:
$$
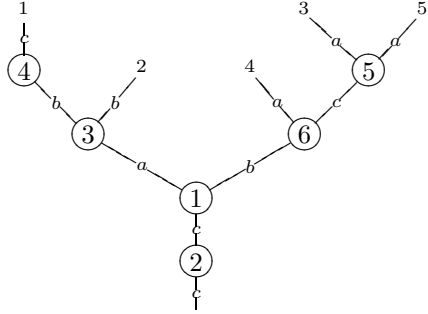
\xymatrixcolsep{1pc}
\xymatrixrowsep{1pc}
\entrymodifiers={=<1pc>[o][F-]}
\xymatrix{
*{^1} \ar@{-}[d]|-c & *{} & *{} & *{} & *{} & *{^3}\ar@{-}[dr]|a & *{} & *{^5}\ar@{-}[dl]|a \\
4 \ar@{-}[dr]|b & *{} & *{^2}\ar@{-}[dl]|b & *{} & *{^4}\ar@{-}[dr]|a & *{} & 5\ar@{-}[dl]|c & *{} \\
*{} & 3 \ar@{-}[drr]|a & *{} & *{} & *{} & 6\ar@{-}[dll]|b & *{} & *{} \\
*{} & *{} & *{} & 1\ar@{-}[d]|-c & *{} & *{} & *{} & *{} \\
*{} & *{} & *{} & 2\ar@{-}[d]|-c & *{} & *{} & *{} & *{} \\
*{} & *{} & *{} & *{} & *{} & *{} & *{} & *{}
}
$$
\end{example}
\begin{proposition}
An algebra over $S^C$ is a $C$\nobreakdash-coloured operad in $\Sets$ and conversely.
\end{proposition}
\begin{proof}
Recall that the set of colours of $S^C$ is 
$$
D=\{(c_1,\ldots, c_n; c)\mid c_i, c\in C, \, n\ge 0\}.
$$
Thus, an $S^C$-algebra is given by a family of sets
$$
P=(P(c_1,\ldots, c_n;c))_{(c_1,\ldots, c_n;c)\in D}
$$
together with a map of $D$-coloured operads $S^C\longrightarrow \End(P)$, i.e., maps of sets
$$
\Phi\colon S^C(\overline{c}_1,\ldots,\overline{c}_n; \overline{a})\longrightarrow \Sets(P(\overline{c}_1)
\times\cdots\times P(\overline{c}_n), P(\overline{a})),
$$
where $\overline{c}_i=(c_{i,1},\ldots, c_{i,k_i}; c_i)$ and $\overline{a}=(a_1,\ldots, a_m;a)$.
In particular, the maps
$$
S^C(\,\,\, ; (c;c))\longrightarrow \Sets(*, P(c;c))
$$
give the units of $P$. If $\alpha\in \Sigma_n$, then the right action
$$
\alpha^*\colon P(c_1,\ldots, c_n;c)\longrightarrow P(c_{\alpha^{-1}(1)},\ldots, c_{\alpha^{-1}(n)};c)
$$
is defined by taking $\Phi(T)$, where $T$ is the tree
$$
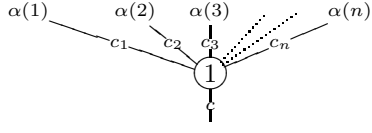
\xymatrixcolsep{1pc}
\xymatrixrowsep{1pc}
\entrymodifiers={=<1pc>[o][F-]}
\xymatrix{
*{^{\alpha(1)}}\ar@{-}[drrr]|{c_1} &  *{} & *{^{\alpha(2)}}\ar@{-}[dr]|{c_2} & *{^{\alpha(3)}}\ar@{-}[d]|{c_3}
& *{}\ar@{.}[dl] & *{}\ar@{.}[dll] & *{^{\alpha(n)}}\ar@{-}[dlll]|{c_n}\\
*{} & *{} & *{} & 1\ar@{-}[d]|-c & *{} & *{} & *{} & *{}\\
*{} & *{} & *{} & *{} & *{} & *{} & *{} & *{}
}
$$
The $\circ_i$-operations
$$
\xymatrix{
P(c_1,\ldots,c_i,\ldots, c_n; c)\times P(a_1,\ldots, a_m; c_i) \ar[d]^{\circ_i} \\
P(c_1,\ldots, c_{i-1},a_1,\ldots, a_m,c_{i+1},\ldots, c_n;c)
}
$$
of $P$ are defined as the image under $\Phi$ of the element
$$
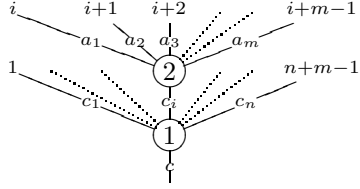
\xymatrixcolsep{1pc}
\xymatrixrowsep{1pc}
\entrymodifiers={=<1pc>[o][F-]}
\xymatrix{
*{^{i}}\ar@{-}[drrr]|{a_1} &  *{} & *{^{i+1}}\ar@{-}[dr]|{a_2} & *{^{i+2}}\ar@{-}[d]|{a_3}
& *{}\ar@{.}[dl] & *{}\ar@{.}[dll] & *{^{i+m-1}}\ar@{-}[dlll]|{a_m}\\
*{^1}\ar@{-}[drrr]|{c_1} &  *{}\ar@{.}[drr] & *{}\ar@{.}[dr] & 2\ar@{-}[d]|{c_i}
& *{}\ar@{.}[dl] & *{}\ar@{.}[dll] & *{^{n+m-1}}\ar@{-}[dlll]|{c_n}\\
*{} & *{} & *{} & 1\ar@{-}[d]|-c & *{} & *{} & *{} & *{}\\
*{} & *{} & *{} & *{} & *{} & *{} & *{} & *{}
}
$$
in
$
S^C((c_1,\ldots, c_n; c),(a_1,\ldots, a_m; c_i);
(c_1,\ldots,c_{i-1},a_1,\ldots, a_m,c_{i+1},\ldots, c_n;c)).
$

Conversely, given a $C$\nobreakdash-coloured operad $Q$, any triple
$(T,\sigma,\tau)$ contained in the set $S^C(\overline{c}_1,\ldots, \overline{c}_n;
\overline{a})$ acts on an element $(e_1,\ldots,e_n)\in
Q(\overline{c}_1)\times\cdots\times Q(\overline{c}_n)$ by labeling
the vertex $\sigma(i)$ of $T$ by $e_i$, and then using the coloured
operad structure of $Q$ to compose $(e_1,\ldots, e_n)$ along the
tree $T$ to obtain an element $e\in Q(\overline{a})$, and then
applying the right action by $\tau$ to this element. 
\end{proof}

\subsection{The coloured operad $S^C$ in monoidal categories}
\label{sc_mon}
The operad $S^C$ can be transported to any closed symmetric monoidal category $\V$ to obtain a coloured operad
in $\V$ whose algebras are $C$\nobreakdash-coloured operads in $\E$, where $\E$ is any monoidal $\V$-category.
More precisely, if $\mathcal{V}$ is any closed symmetric monoidal category,
then the strong symmetric monoidal functor $(-)_{\mathcal{V}}\colon
\Sets\longrightarrow \mathcal{V}$ defined as
$$
A_{\mathcal{V}}=\coprod_{x\in A}I
$$
for every set $A$ sends coloured operads to coloured operads. Hence, by applying
this functor to the coloured operad $S^C$ in $\Sets$, we obtain another coloured operad
$S^C_{\mathcal{V}}$ in $\mathcal{V}$ whose algebras in $\V$ are precisely
$C$\nobreakdash-coloured operads in $\mathcal{V}$.

If $\mathcal{E}$ is a monoidal $\V$-category, then coloured operads in $\mathcal{V}$ act on
$\mathcal{E}$. Thus, $S^C_{\mathcal{V}}$ is a coloured operad in
$\mathcal{V}$ whose algebras in $\mathcal{E}$ are $C$\nobreakdash-coloured
operads in~$\mathcal{E}$.

\begin{remark}
We will use $S^C_{\mathcal{V}}$ to obtain a model structure for 
coloured operads in symmetric spectra. There is an alternative
approach: the category of $C$\nobreakdash-coloured collections is strictly
monoidal \cite[Appendix]{BM07} and the monoids in this category are
the $C$\nobreakdash-coloured operads, i.e., a $C$\nobreakdash-coloured operad is an algebra over
the non-symmetric operad encoding monoid structures. We have chosen the
present way for two reasons. Firstly, we believe that the operad 
$S^C_{\mathcal{V}}$ is of interest in its own right. Secondly, we cannot directly use the 
results of Berger and Moerdijk \cite{BM03,BM07}, who only work in
\emph{symmetric} monoidal categories, since the
category of $C$\nobreakdash-coloured collections is \emph{not} symmetric monoidal.
\end{remark}

\section{Coloured operads with values in symmetric spectra}
We recall from \cite{BM07} the basic properties and terminology of a
model structure for the category of coloured operads in a symmetric
monoidal model category. We assume that our model categories have
functorial factorizations, as in \cite{Hov99} and \cite{Hir03}.

\subsection{Enriched model categories}
\label{model}
A \emph{monoidal model category} $\V$  is a closed symmetric monoidal category with a model structure satisfying the
following conditions:

\begin{itemize}
\item[(i)] \emph{Pushout-product axiom}. If $f\colon A\longrightarrow B$ and $g\colon U\longrightarrow V$ are two cofibrations in $\V$, then
the induced map
\begin{equation}
(A\otimes V)\coprod_{A\otimes U} (B\otimes U)\longrightarrow B\otimes V
\label{pushproduct}
\end{equation}
is a cofibration that is trivial if $f$ or $g$ is trivial. A direct consequence of the pushout-product
axiom is that tensoring with a cofibrant object preserves cofibrations (and trivial cofibrations), and that the tensor product
of two cofibrations with cofibrant domains is again a cofibration. 
\item[(ii)] \emph{Unit axiom}. The natural map
$$
q\otimes 1\colon QI\otimes A \longrightarrow I\otimes A
$$
is a weak equivalence in $\V$ for every cofibrant $A$, where $Q$
denotes the cofibrant replacement functor of $\V$. The unit axiom holds
trivially if the unit of $\V$ is cofibrant.
\end{itemize}

Let $\V$ be a monoidal model category. A \emph{model category enriched over $\V$} is a category $\E$ enriched over $\V$ with
a model structure such that the enrichment $\Hom_{\E}(-,-)$ satisfies an analog of Quillen's (SM7) axiom for
simplicial categories; namely, if $f\colon X\longrightarrow Y$ is a cofibration in $\E$ and $g\colon W\longrightarrow Z$ is
a fibration in $\E$, then the induced map
\begin{equation}
\Hom\nolimits_{\E}(Y,W)\longrightarrow\Hom\nolimits_{\E}(Y,Z)\times_{\Hom_{\E}(X,Z)}\Hom\nolimits_{\E}(X, W)
\label{enrichedcond}
\end{equation}
is a fibration in $\V$ that is trivial if $f$ or $g$ is trivial.

Recall from \cite[Definition 4.2.18]{Hov99} that a \emph{$\V$-module model category} is a $\V$-module category $\E$ with a model structure such that:
\begin{itemize}
\item[(i)] The functor $-\otimes-\colon \V\times \E\longrightarrow \E$ is a Quillen bifunctor, that is, the pushout\nobreakdash-product of a cofibration in $\V$ and a cofibration in $\E$ is a cofibration in $\E$. 
\item[(ii)] The map $q\otimes 1\colon QI\otimes X\longrightarrow I\otimes X$ is a weak equivalence in $\E$ for every cofibrant object $X$ in $\E$.
\end{itemize}

A \emph{monoidal $\V$-model category} is a $\V$-module model category $\E$ that is also a monoidal model category and such that the $\V$-action commutes with the monoidal product of $\E$ (see \cite[\S11.3.3, \S11.3.4]{Fre09}), i.e., there are natural coherent isomorphisms
$$
A\otimes (X\otimes Y)\cong (A\otimes X)\otimes Y
$$
for every $X$, $Y$ in $\E$ and every $A$ in $\V$. Any monoidal model category $\E$ is a monoidal $\E$-model category.

\subsection{Model structures for coloured operads}
For any cofibrantly generated monoidal model category $\V$, the
category of $C$\nobreakdash-coloured collections $\Coll_C(\V)$ admits a
cofibrantly generated  model structure in which a morphism
$K\longrightarrow L$ is a weak equivalence or a fibration if for
each $(n+1)$-tuple of colours $(c_1,\ldots, c_n;c)$, the map
$$
K(c_1,\ldots, c_n; c)\longrightarrow L(c_1,\ldots, c_n;c)
$$
is a weak equivalence or a fibration, respectively, in $\V$. This
model structure can be transferred along the free-forgetful
adjunction~(\ref{freeforgetful}),
$$
\xymatrix{
F: \Coll_C(\V) \ar@<3pt>[r] & \ar@<3pt>[l] {\rm Oper}_C(\V): U,
}
$$
to provide a model structure on the category of $C$\nobreakdash-coloured operads
in $\V$, if any sequential colimit of pushouts of images under $F$
of the generating trivial cofibrations of $\Coll_C(\V)$ yields a
weak equivalence in $\Coll_C(\V)$ after applying the forgetful
functor (cf.\ \cite[\S 3]{Cra95}, \cite[Theorem 11.3.2]{Hir03}). In
this model structure, a morphism of $C$\nobreakdash-coloured operads $f\colon
P\longrightarrow Q$ is a weak equivalence or a fibration if and only
if $U(f)$ is a weak equivalence or a fibration of
$C$\nobreakdash-coloured collections, respectively.

Although this assumption is usually hard to verify, it is satisfied if we assume that the unit of $\V$ is
cofibrant, the existence of a symmetric monoidal fibrant replacement functor for $\V$, and the existence of an interval
with a coassociative and cocommutative comultiplication \cite[Theorem 2.1]{BM07}. Recall that a \emph{symmetric monoidal fibrant replacement functor} in $\V$ is a fibrant replacement functor $F$ that is symmetric monoidal and such that for every $X$ and $Y$ in $\V$ the following diagram commutes:
$$
\xymatrix{
X\otimes Y\ar[r]^{r_{X\otimes Y}}\ar[d]_{r_X\otimes r_Y} & F(X\otimes Y) \\
FX\otimes F(Y),\ar[ur] &
}
$$ 
where $r\colon \rm{Id}_{\V}\longrightarrow F$ is the natural transformation associated with the fibrant replacement.

The categories of simplicial sets or
$k$-spaces (with the Quillen model structure), among others, satisfy the conditions above.

\subsection{Coloured operads in symmetric spectra}
We denote by $\Sp$ the category of symmetric spectra. When we refer to its model structure, we will
understand it as the \emph{positive model structure}, as described
in \cite{MMSS01} or in \cite{Shi04}. The weak equivalences are the
usual stable weak equivalences, and positive cofibrations are stable
cofibrations with the additional assumption that they are
isomorphisms in level zero. Positive fibrations are defined by the
right lifting property with respect to the trivial positive
cofibrations. With this model structure, the category of symmetric
spectra is a cofibrantly generated proper monoidal model category.

However, the assumptions required in \cite[Theorem 2.1]{BM07} are
not satisfied for this category. In the category of symmetric
spectra with the positive model structure, the unit spectrum $S$ is
not cofibrant, and the existence of a symmetric monoidal fibrant
replacement functor is not known. In fact, no symmetric monoidal model category of spectra can
simultaneously have a cofibrant unit and a symmetric monoidal
fibrant replacement functor~\cite{Lew91}. In this section we show how we can
avoid this problem by using the coloured operad described in
Section~\ref{colopasalg}.

Let $\E$ be a monoidal $\V$-model category and $P$ any $C$\nobreakdash-coloured operad in $\V$. Then
there is an adjoint pair
$$
\xymatrix{
F_P : \E^C \ar@<3pt>[r] & \ar@<3pt>[l] {\rm Alg}_P(\E) : U_P,
}
$$
where $F_P$ is the free $P$-algebra functor defined as
$$
F_P({\bf X})(c)=\coprod_{n\ge 0}\left(\coprod_{c_1,\dots, c_n\in C}P(c_1,\ldots, c_n;c)
\otimes_{\Sigma_n} X(c_1)\otimes\cdots\otimes X(c_n) \right)
$$
for every ${\bf X}=(X(c))_{c\in C}$ in $\E^{C}$, and $U_P$ is the
forgetful functor. Following the terminology of \cite{BM07}, we say that a $C$\nobreakdash-coloured
operad $P$ in $\V$ is \emph{admissible} in $\E$ if the model
structure on $\E^C$ is transferred to ${\rm Alg}_P(\E)$ along this
adjunction. Thus, if $P$ is admissible, then ${\rm Alg}_P(\E)$ has a
model structure where a map of $P$-algebras
$\mathbf{f}\colon\mathbf{X}\longrightarrow \mathbf{Y}$ is a weak
equivalence or a fibration if and only if
$$
f_c\colon X(c)\longrightarrow Y(c)
$$
is a weak equivalence or a fibration, respectively, in $\E$ for every
$c\in C$.

The category of symmetric spectra $\Sp$ with the positive model
structure is a monoidal $\sSets$-model category, where $\sSets$
denotes the category of simplicial sets with the usual model
structure. We recall the following admissibility result from
\cite[Theorem 1.3]{EM06}:
\begin{theorem}
If we consider the positive model structure in $\Sp$, then any coloured operad $P$ in simplicial sets is
admissible in $\Sp$. $\hfill\qed$
\label{e-m}
\end{theorem}
Using the coloured operad of Section 3, we obtain a model structure
for coloured operads in symmetric spectra.
\begin{corollary}
Let $C$ be a fixed set of colours. Then the category of $C$\nobreakdash-coloured operads in symmetric spectra admits a model
structure in which the weak equivalences are the colourwise stable equivalences and the fibrations are the
colourwise positive stable fibrations of symmetric spectra.
\label{maincor}
\end{corollary}
\begin{proof}
Consider the coloured operad $S^C_{\V}$ of Section 3, where $\V$ is now the category of simplicial sets.
The category of $C$\nobreakdash-coloured operads in $\Sp$ is the category of $S^C_{\V}$-algebras in $\Sp$, and the latter has a model structure by Theorem~\ref{e-m}.
\end{proof}

\begin{remark}
Recall that if $\E$ is a cocomplete category and $\mathcal{I}$ is a set of maps in $\E$, then
the subcategory of relative $\mathcal{I}$-cell complexes is the subcategory of maps that can be
constructed as transfinite compositions of pushouts of elements of $\mathcal{I}$. The admissibility of
every coloured operad with values in simplicial sets in the category of symmetric spectra is based
on the fact that every relative $F_P(\mathcal{J})$-complex is a stable equivalence, where $\mathcal{J}$ is the set
of generating trivial cofibrations of $\E^C$ with $\E=\Sp$ (see \cite[Lemma 11.7]{EM06}). In fact, Theorem~\ref{e-m} and its
corollary remain valid for any cofibrantly generated monoidal $\V$-category $\E$ with this property.
\end{remark}

\begin{remark}
In a recent paper \cite{Kro}, Kro proved that the category of
reduced operads and the category of positive operads in orthogonal
spectra with the positive model structure admits a transferred model
structure, by constructing a symmetric monoidal fibrant replacement
functor. Recall that an operad $P$ in a monoidal category $\V$ is
\emph{positive} if $P(0)=0$, and it is \emph{reduced} if $P(0)=I$,
where $I$ is the unit of $\V$.

In fact, he extended the results of \cite{BM03} to reduced and
positive operads in a monoidal $\V$-model category $\E$
where the unit is not necessarily cofibrant (although the symmetric
monoidal fibrant replacement functor assumption is still needed),
but $\V$ has a nice unit and Hopf intervals (see \cite[Theorem
2.4]{Kro} for an explicit statement).
Observe that we cannot apply the result of Kro in our case, since we
do not have a symmetric monoidal fibrant replacement functor, and
his candidate for orthogonal spectra is not valid for symmetric
spectra (see \cite[Remark 3.4]{Kro}).
\end{remark}

\section{Localization of module structures}

\subsection{Enriched homotopical localization}
Let $\E$ be a $\V$-module model category as in Section~\ref{model}. For all $X$ and $Y$ in $\E$, we define
$$
\hom_{\E}(X,Y)=\Hom\nolimits_{\E}(QX, FY),
$$
where $Q$ denotes a functorial cofibrant replacement and $F$ denotes a functorial fibrant replacement. We call the object
$\hom_{\E}(-,-)$ a \emph{homotopy $\V$-function complex}. Note that $\hom_{\E}(-,-)$ is always fibrant in $\V$
by~(\ref{enrichedcond}). A morphism $X\longrightarrow Y$ and an object $Z$ in $\E$ are called \emph{$\V$-orthogonal} if
the induced map
$$
\hom_{\E}(Y,Z)\longrightarrow \hom_{\E}(X,Z)
$$
is a weak equivalence in $\V$.

The following definition extends~\cite[Definition 4.1]{CGMV}.
\begin{definition}
An \emph{enriched homotopical localization} in a $\V$-module model category $\E$ is a functor $L\colon\E\longrightarrow \E$ that
preserves weak equivalences and takes fibrant values, together with a natural transformation
$\eta\colon{\rm Id}_{\E}\longrightarrow L$ such that, for every $X$ in $\E$, the following hold:
\begin{itemize}
\item[{\rm (i)}] $L\eta_X\colon LX\longrightarrow LLX$ is a weak equivalence in $\E$.
\item[{\rm (ii)}] $\eta_{LX}$ and $L\eta_X$ are equal in the homotopy category ${\rm Ho}(\E)$.
\item[{\rm (iii)}] For every $X$ in $\E$, the map $\eta_X\colon X\longrightarrow LX$ is a cofibration such that the induced
map
$$
\hom_{\E}(LX,LY)\longrightarrow \hom_{\E}(X,LY)
$$
is a weak equivalence in $\V$ for all $Y$ in $\E$.
\end{itemize}
\end{definition}

If $L$ is an enriched homotopical localization, then the fibrant objects of $\E$ weakly equivalent to $LX$ for some $X$
are called \emph{$L$-local}, and the maps $f$ such that $Lf$ is a weak equivalence are called \emph{$L$-equivalences}.
It follows using~(\ref{enrichedcond}) and Ken Brown's lemma \cite[Lemma 1.1.12]{Hov99} that $L$-local objects and $L$\nobreakdash-equivalences are $\V$-orthogonal. In fact, a fibrant object is $L$-local if and only if it
is $\V$-orthogonal to all $L$-equivalences, and a map is an $L$-equivalence if and only if it is $\V$-orthogonal to all
$L$-local objects.

The main source for enriched homotopical localizations comes from
enriched left Bousfield localizations~\cite{Bar10}. Enriched left
Bousfield localizations are similar to Bousfield localizations, but
defining the class of local objects by means of
$\V$-orthogonality instead of simplicial orthogonality.
As proved in \cite[Theorem~4.46]{Bar10}, the enriched left Bousfield
localization with respect to a set of morphisms always exists in a
$\V$-model category $\E$, provided that the category $\E$
is cotensored over $\V$, left proper and combinatorial, and the category $\V$ is
combinatorial. This is the case, for example,  of the category of symmetric spectra enriched over itself or 
enriched over simplicial sets.

The following theorem is our main result on the preservation of algebras over operads in monoidal $\V$-categories. For simplicity, we state it for
coloured operads with only one colour, but it can be generalized to $C$\nobreakdash-coloured operads by using ideals on the
set of colours $C$; cf. \cite[Theorem 6.1]{CGMV}.

Given two $P$-algebra structures $\gamma$ and $\gamma'\colon P\longrightarrow\End(\mathbf{\mathbf{X}})$ on an object $\mathbf{X}$ of $\E^C$, we say that $\gamma$ and $\gamma'$ \emph{coincide up to homotopy} if they are equal in the homotopy category of $C$-coloured operads.

\begin{theorem}
Let $(L,\eta)$ be an enriched homotopical localization on a mo\-noi\-dal $\V$-model category $\E$ such that the
category of operads in $\V$ admits a transferred model structure.
Let $P$ be a cofibrant operad in $\V$ and let $X$ be a $P$-algebra such that
$X$ is cofibrant in $\E$. Suppose that, for every $n\ge 1$, the map
\begin{equation}
\eta_X^{\otimes n}\colon X\otimes\stackrel{(n)}{\cdots}\otimes X\longrightarrow LX\otimes\stackrel{(n)}{\cdots}\otimes LX
\label{displayed_cond}
\end{equation}
is an $L$-equivalence. Then $LX$ admits a homotopy
unique $P$-algebra structure such that $\eta_X$ is a map of $P$-algebras.
\label{mainthm}
\end{theorem}
\begin{proof}
The map $\eta_X^{\otimes n}$ is a cofibration for every $n\ge 1$, since it is a tensor product of $n$ cofibrations with cofibrant
domains in a monoidal model category. Hence the induced map
$$
\xymatrix{
\Hom_{\E}(LX\otimes\stackrel{(n)}{\cdots}\otimes LX, LX)\ar[d] \\
\Hom_{\E}(X\otimes\stackrel{(n)}{\cdots}\otimes X, LX)
}
$$
is a trivial fibration for every $n\ge 1$. Indeed, it is a weak equivalence because $LX$ is $L$-local and $\eta_X^{\otimes n}$ is an $L$-equivalence
by assumption, and it is also a fibration by~(\ref{enrichedcond}). Thus, the morphism of collections
$$
\End(LX)\longrightarrow \Hom(X, LX)
$$
induced by $\eta_X$ is a trivial fibration.

Now consider the endomorphism operad ${\rm End}_P(\eta_X)$ associated to $\eta_X$. It is obtained as the following
pullback of collections:
\begin{equation}
\label{mainpullback}
\xymatrix{
{\rm End}(\eta_X) \ar@{.>}[r]^{\alpha} \ar@{.>}[d]_{\beta} & {\rm End}(LX) \ar[d] \\
{\rm End}(X) \ar[r] & {\rm Hom}(X, LX).
}
\end{equation}
The operad $P$ is cofibrant by hypothesis, and the map $\beta$ is a trivial fibration since it is a pullback of a trivial
fibration. Therefore, there is a lifting
$$
\xymatrix{
& {\rm End}(\eta_X) \ar[d] \\
P \ar[r] \ar@{.>}[ur] & {\rm End}(X)
}
$$
where $P\longrightarrow {\rm End}(X)$ is the given $P$-algebra structure of $X$. The
$P$-algebra structure on $LX$ is obtained by composing this lifting with the upper morphism
$\alpha$ in~(\ref{mainpullback}), and with this structure $\eta_X$ is a map of $P$-algebras.

To prove uniqueness, suppose that we have two $P$-algebra structures on $LX$, denoted by
$\gamma,\gamma'\colon P\longrightarrow{\rm End}(LX)$, and assume that $\eta_X$
is a map of $P$\nobreakdash-algebras for each
of them, i.e., $\gamma$ and $\gamma'$ factor through ${\rm End}(\eta_X)$.
Now let $\delta,\delta'\colon P\longrightarrow {\rm End}(\eta_X)$ be such that
$\gamma=\alpha\circ\delta$ and $\gamma'=\alpha\circ\delta'$, with $\alpha$ as in~(\ref{mainpullback}).
Since $\beta\circ\delta=\beta\circ\delta'$ and $\beta$ is a trivial fibration, it follows that
$\delta$ and $\delta'$ are equal in the homotopy category of operads and  hence
so are $\gamma$ and $\gamma'$.
\end{proof}

\begin{remark}
The displayed condition~(\ref{displayed_cond}) holds automatically when $\V=\E$, that is, when $\E$ is a monoidal model category and the localization is enriched in $\E$. 
\end{remark}

\subsection{Localization of module spectra}
Let $\E$ be a monoidal $\V$-model category. Recall  that the operad $\mathcal{A}ss$ in $\V$ whose algebras in $\E$ are the associative monoids is defined by 
$\mathcal{A}ss(n)=I[\Sigma_n]$ for $n\ge 0$, where $I[\Sigma_n]$ denotes a coproduct of copies of the unit $I$ indexed by $\Sigma_n$ on which 
$\Sigma_n$ acts freely by permutations. Suppose that the operad $\mathcal{A}ss$ and the coloured operad $S_{\V}^C$ in $\V$, with $C$ a set with one element, are admissible in $\E$.
Then, the category ${\rm Mon}(\E)$ of monoids in $\E$ and the category $\Oper(\E)$ of (one-coloured) operads in
$\E$ have a transferred model structure, since they are categories of algebras over ${\mathcal A}ss$ and $S_{\V}^C$,
respectively. These two categories are related by a pair of adjoint functors
$$
\xymatrix{
P_{(-)} : {\rm Mon}(\E)\ar@<3pt>[r] & \ar@<3pt>[l]  \Oper(\E) : (-)(1),
}
$$
where the left adjoint sends a monoid $R$ to the operad $P_R$ defined by
$$
P_R(n)=
\left\{
\begin{array}{ll}
R & \mbox{if $n=1$,}\\
0 & \mbox{otherwise,}
\end{array}
\right.
$$
and the right adjoint sends any operad $Q$ to the monoid $Q(1)$. In fact, this adjoint pair is a Quillen pair, since
the right adjoint preserves fibrations and weak equivalences.

The following result is an application of Theorem~\ref{mainthm} to the preservation of module structures under enriched homotopical
localizations.

\begin{theorem}
Let $(L,\eta)$ be an enriched homotopical localization on a mo\-noi\-dal model category $\E$ such that
the category of operads in $\E$ and the category of monoids in $\E$ admit a transferred model structure.
Let $R$ be a cofibrant monoid in $\E$ and let $X$ be an $R$-module such that
$X$ is cofibrant in $\E$. Then $LX$ admits a homotopy
unique $R$-module structure such that $\eta_X$ is a map of $R$-modules.
\label{locmodules}
\end{theorem}
\begin{proof}
If $R$ is a monoid in $\E$, then an $R$-module is the same as an
algebra over the operad $P_R$. Note that, if $R$ is a cofibrant
object in ${\rm Mon}(\E)$, then the operad $P_R$ is also cofibrant
in ${\Oper}(\E)$, since $P_{(-)}$ is a left Quillen functor. The
result now follows from Theorem~\ref{mainthm}, since the operad
$P_R$ is concentrated in valence one and considering $\E$ itself as a monoidal $\E$-category.
\end{proof}

Now we make this result explicit for module spectra in the category of symmetric spectra $\Sp$ with the positive model structure.  We will consider enriched homotopical localizations in $\Sp$ viewed as a monoidal $\Sp$-category, i.e., using the enrichment given by the internal hom.

\begin{corollary}
Let $(L,\eta)$ be an enriched homotopical localization on the category of symmetric spectra (enriched over itself). 
Let $M$ be a module over a cofibrant ring symmetric spectrum $R$ and assume that $M$ is cofibrant as a
symmetric spectrum. Then $LM$ has a homotopy unique module structure over $R$ such that $\eta_M\colon M\longrightarrow LM$ is
a morphism of $R$-modules.
\end{corollary}
\begin{proof}
The category $\Sp$ is a monoidal $\sSets$-category, where $\sSets$ denotes the category of simplicial sets and any coloured operad in $\sSets$ is 
admissible in $\Sp$ by Theorem~\ref{e-m}. Thus, the category of operads in $\Sp$ and the category of ring symmetric spectra both admit a transferred model structure. The result now follows directly from Theorem~\ref{locmodules}.
\end{proof}

\end{document}